%% file: main.tex
\documentclass{article}
\usepackage{amsmath,enumitem}
\input{mymakros.tex}
\input{settings}

\title{Stable Extensions of Complete Groups}
\author{Isaac Ochoa}
\date{}

\begin{document}

\maketitle

\begin{abstract} A group is said to be stable if it is isomorphic to its automorphism group. We investigate how we can extend centerless groups to construct finite stable groups with nontrivial centers. To this end, we classify all finite stable groups arising as central extensions of centerless groups. Furthermore, all finite stable groups arising as extensions of centerless groups by groups of nilpotency class two with trivial induced outer action on the kernel are classified. 
\end{abstract}

\begin{section}{Introduction}
Passing from a group to the group of its automorphisms is a well-motivated operation that nonetheless yields a striking amount of complexity to consider. Even in the finite case, the eventual behavior of the so-called automorphism towers generated by repeated application of this operation is not well understood. An old result of Wielandt, given in \cite{Wielandt1939}, demonstrates that for a finite centerless group, all but finitely many of the groups appearing in its automorphism tower will be isomorphic. One may hope that in understanding what kinds of groups can appear at the top of an automorphism tower with eventually stabilizing behavior, some insight can be gained into the mechanisms by which an automorphism tower might stabilize (or fail to do so) for a finite group with nontrivial center. Groups isomorphic to their automorphism groups, henceforth referred to as stable groups, therefore offer a fascinating lens through which to approach the endeavor of better understanding the interplay between a group's structure and its automorphisms. In general there is currently very little that can be said about an arbitrary stable group, due in large part to the dearth of information gleaned from the mere existence of an isomorphism inducing stability. An arbitrary complete group $G$, on the other hand, comes equipped with an explicit isomorphism $c:G\rightarrow \Aut(G)$ given by $c(g)=c_g$, where $c_g(x)=gxg^{-1}$. As such, there is a much more robust understanding of complete groups than in the general case. Completeness, however, is a much stronger property, being equivalent to the nonexistence of nontrivial central elements or outer automorphisms. This leaves a large rift between the constrained structure of complete groups and the unbridled generality of arbitrary stable groups to be explored. Some results have excluded the existence of non-complete stable groups under the presence of additional assumptions. In particular, the nonexistence of many kinds of finite stable $p$-groups is established by Cutolo in \cite{CutoloD4}, although it remains an open problem (originally posed by A. Mann, see \cite[Problem 15.29]{KNotebook}) whether there are any finite stable $p$-groups other than the dihedral group of order eight. Despite this, little work has been devoted to the construction of non-complete stable groups. The pursuit of such constructions will serve as the focus of this paper. We will leverage the structure of complete groups, considering stable groups obtained as extensions of complete groups. 

In Section 3 we establish that finite stable groups arising as extensions of centerless groups by nilpotent groups whose induced outer action on the kernel is trivial must be isomorphic to the direct product of the kernel and quotient, and further that the quotient is itself complete (Lemma \ref{split}). As such, while we aim to exploit the structure of complete groups in our constructions, we need only assume a centerless quotient to do so. Lemma \ref{split} is then applied to classify all finite stable groups arising as central extensions of centerless groups (Theorem \ref{abthm}). We find that such extensions must have kernel of order two, and give necessary and sufficient conditions on the structure of the necessarily complete quotient for such an extension to yield a stable group. In Section 4 we classify all finite groups arising as extensions of centerless groups by groups of nilpotency class two whose induced outer action on the kernel is trivial (Theorem \ref{nilthm}), again using Lemma \ref{split} to conclude a direct product structure and the completeness of the centerless factor. We find that the nilpotent factor must be isomorphic to dihedral group of order eight, and again establish necessary and sufficient conditions on the structure of the complete factor for such a product to yield a stable group. To the knowledge of the author, the groups classified by Theorems \ref{abthm} and \ref{nilthm} constitute all known non-complete finite stable groups.
\end{section}

\section{Notation \& Background}

In this section we give pertinent definitions and outline the notation to be used throughout the paper. All groups are assumed to be finite.

Given a group $G$ and an element $g \in G$, conjugation by $g$ is represented by $c_g$ throughout, where $c_g(x)=gxg^{-1}$ for all $x \in G$. A group $G$ is said to be stable if $G\cong \Aut(G)$. Recall a group $G$ is complete if the conjugation map $c:G\rightarrow \Aut(G)$, defined by $c(g)=c_g$, is an isomorphism. 

Let $n \in \bbZ_{\geq 1}$. We use $C_n$ to denote the cyclic group of order $n$ and $D_n$ to denote the dihedral group of order $2n$. Euler's totient function is represented by $\phi$, where \[\phi(n)=|\{k \in \bbZ_{\geq 1} \mid k <n \textrm{ and } \gcd(k,n)=1\}|.\] For any prime $p$, we write $\upsilon_p(n)$ to denote the $p$-adic valuation of $n$, the exponent of the highest power of $p$ dividing $n$.

We write $G'$ and $G^{\ab}$ to denote the derived subgroup and abelianization of $G$, respectively. Given a subgroup $N\leq G$, we write $C_G(N)$ to denote the centralizer of $N$ in $G$. The upper-central series of a group $G$ is denoted \[\{e_G\}=Z_0(G)\norm Z_1(G)\norm Z_2(G)\norm \cdots,\] where $Z_{i+1}(G)=\{x \in G\mid [x,y]\in Z_i(G) \textrm{ for all } y \in G\}$. 

Finally, recall that for a finite group $G$, the following are equivalent: 
\begin{enumerate}[label=\roman*)]
    \item $G$ is nilpotent.
    \item $Z_i(G)=G$ for some $i \in \bbZ_{\geq 0}$.
    \item $G$ is an internal direct product of its Sylow $p$-subgroups.
\end{enumerate}

\section{Stable Central Extensions of Complete Groups}

Given a centerless group $K$, it is natural to ask how we may extend $K$ to a form a stable group with nontrivial central elements. To this end, we begin this section by establishing a direct product structure for stable groups arising as extensions of centerless groups by nilpotent groups whose induced outer action on the kernel is trivial, and further establish completeness of the centerless factor. We then proceed with a consideration of the consequences of an abelian kernel, and obtain as a result the classification of all stable groups arising as central extensions with centerless quotients.

\begin{thm} 
\label{extchar}
Let $1\rightarrow N\rightarrow G\xrightarrow{\pi} K\rightarrow 1$ be an extension such that the induced outer action on $N$ is trivial. Identifying $N$ with its image in $G$, if $N$ is nilpotent and $K$ is centerless, then $Z_i(N)=Z_i(G)$ for all $i \in \bbZ_{\geq 0}$
and both $N$ and $C_G(N)$ are characteristic in $G$.
\end{thm}

\begin{proof}
As the outer action on $N$ is trivial we have $G=NC_G(N)$. We will show by induction that $Z_i(G)=Z_i(N)$ for all $i \in \bbZ_{\geq 0}$. Suppose $Z_i(N)=Z_i(G)$ for some $i \in \bbZ_{\geq 0}$. Let $z \in Z_{i+1}(G)$ and $k \in K$. Note that $k=\pi(g_k)$ for some $g_k \in G$. Then $[z,g_k] \in Z_i(G)=Z_i(N)$. As $Z_i(N)\leq N=\ker \pi$, we have 
\begin{align*}
[\pi(z),k]=[\pi(z),\pi(g_k)]=\pi([z,g_k])=e_K.
\end{align*} Thus $\pi(z) \in Z(K)$ as $k$ is arbitrary. However, $K$ is centerless, so $\pi(z)=e_K$, and thus $z \in \ker\pi=N$. Let $\ell \in N$. Then $[z,\ell] \in Z_i(G)$ as $z \in Z_{i+1}(G)$. Further, $Z_i(G)=Z_i(N)$ by assumption, so $z \in Z_{i+1}(N)$ as $\ell$ is arbitrary and $z \in N$. Thus, as $z$ is arbitrary,  $Z_{i+1}(G)\leq Z_{i+1}(N)$. Let $m \in Z_{i+1}(N)$ and $h \in G$. As the outer action on $N$ is trivial, $c_{h}(m^{-1})=c_n(m^{-1})$ for some $n \in N$. Thus 
\begin{align*}
[m,h]=mc_{h}(m^{-1})=mc_{n}(m^{-1})=[m,n] \in Z_{i}(N).
\end{align*} By assumption, $Z_{i}(N)=Z_{i}(G)$, so $m \in Z_{i+1}(G)$ as $h$ is arbitrary. Since $m$ is arbitrary, $Z_{i+1}(N)\leq Z_{i+1}(G)$, so $Z_{i+1}(N)=Z_{i+1}(G)$. Finally, observe that we clearly have $Z_0(N)=Z_0(G)$. Therefore, by induction we have $Z_i(G)=Z_i(N)$ for all $i \in \bbZ_{\geq 0}$. 

As $N=Z_j(G)$ for some $j \in \bbZ_{\geq 0}$, $N$ is characteristic in $G$, so $C_G(N)$ is characteristic in $G$.
\end{proof}

\begin{lemma}\label{split} Let $1\rightarrow N\rightarrow G\xrightarrow{\pi} K\rightarrow 1$ be an extension such that the induced outer action of $K$ on $N$ is trivial. If $N$ is nilpotent, $K$ is centerless, and $G$ is stable, then $G\cong N\times K$ and $K$ is complete.

\end{lemma}

\begin{proof} Identify $N$ with its image in $G$, and observe that $G$ is a central product of $N$ and $C_G(N)$. Define \[F=\Aut(N)\times_{\Aut(Z(G))}\,\Aut(C_G(N)).\] As $N$ and $C_G(N)$ are characteristic in $G$ by Lemma \ref{extchar}, we have $\Aut(G)\cong F$. Let $L=\{\id_{N}\}\times \Inn(C_G(N))$ and observe that $L \norm F$. As $G$ is stable, $F \cong \Aut(G)\cong G$. Therefore, there is a normal subgroup $\widetilde{K}\norm G$ with $L \cong \widetilde{K}$. Furthermore, 
$L\cong \Inn(C_G(N))\cong K$, so $K\cong\widetilde{K}$. 

Suppose towards contradiction $N\cap \widetilde{K}\neq \{e_G\}$. As $N$ is nilpotent and $N\cap \widetilde{K}\norm N$, we have $Z(N)\cap (N\cap \widetilde{K})\neq \{e_G\}$, so $Z(N)\cap \widetilde{K}\neq \{e_G\}$. By Lemma \ref{extchar}, $Z(N)=Z(G)$, so $Z(G)\cap \widetilde{K}\leq Z(\widetilde{K})$ is nontrivial, which is clearly a contradiction as $\widetilde{K}\cong K$ is centerless. Thus $N\cap \widetilde{K}=\{e_G\}$, so $|N\widetilde{K}|=|N|\cdot |\widetilde{K}|=|G|$, giving $G=N\widetilde{K}$. As $N\norm G$ and $\widetilde{K}\norm G$, we find $G\cong N\times \widetilde{K}\cong N \times K$.

Since $K$ is isomorphic to a direct factor of $G$, there is a subgroup $M\leq \Aut(G)$ such that $M\cong \Aut(K)$. As $\Aut(G)\cong G$, we further have a subgroup $\widetilde{M}\leq G$ with $\widetilde{M}\cong M\cong \Aut(K)$. As $K$ is centerless, it follows that $\Aut(K)\cong \widetilde{M}$ is centerless. Let $j$ be the nilpotency class of $N$, and note $Z_j(G)=N$ by Lemma \ref{extchar}. Suppose towards contradiction there exists a nontrivial element $m_j \in \widetilde{M}\cap Z_j(G)=\widetilde{M}\cap N$. As $\widetilde{M}$ is centerless and $m_j \in \widetilde{M}$, there exists an element $x \in \widetilde{M}$ such that $[m_j,x]\neq e_G$. As $\widetilde{M}$ is a subgroup and $m_j \in Z_j(G)$, we have $[m_j,x] \in \widetilde{M}\cap Z_{j-1}(G)$. Setting $m_{j-1}=[m_j,x]$ and continuing in this manner, we obtain a sequence of nontrivial elements $m_{j},m_{j-1},\dots, m_1$. However $m_1 \in \widetilde{M}\cap Z(G)=\{e_G\}$, which is a contradiction as $m_1$ was constructed to be nontrivial. Thus no such $m_j$ exists, so $|\widetilde{M}\cap N|=1$. 

We therefore see that the natural projection of $\widetilde{M}$ onto $G/N$ has trivial kernel, so $|\Aut(K)|=|\widetilde{M}|\leq |G/N|=|K|$. Since $K$ is centerless, we further have $|K|=|\Inn(K)|\leq |\Aut(K)|$, so $|\Inn(K)|=|\Aut(K)|$. Therefore $\Inn(K)=\Aut(K)$, so $K$ is complete as $K$ is centerless, as desired.
\end{proof}

\begin{lemma} 
\label{decomp}
 Let $G=N \times K$ where $K$ is a complete group and $N$ shares no direct factors with $K$. Then $\Aut(G)$ is isomorphic to  $(\Hom(K^{\ab},Z(N))\rtimes \Aut(N))\times K$ where $\Aut(N)$ acts on $\Hom(K^{\ab},Z(N))$ by post-composition.
\end{lemma}

\begin{proof}
Let $*$  denote the group multiplication of $(\Hom(K^{\ab},Z(N))\rtimes \Aut(N))\times K$. Define \[\Phi:(\Hom(K^{\ab},Z(N))\rtimes \Aut(N))\times K\rightarrow \Aut(G)\] by $\Phi((\alpha,\beta,\gamma))=\varphi$, where $\varphi((n,k))=(\beta(n)\alpha(kK'),c_{\gamma}(k))$. One quickly checks that $\Phi$ is a well-defined, i.e. that each such $\varphi$ is automorphism. We aim to show $\Phi$ is an isomorphism. Let $(\alpha,\beta,\gamma),(\lambda,\mu,\nu) \in (\Hom(K^{\ab},Z(N))\rtimes \Aut(N))\times K$ and $(n,k) \in G$. We have 
\begin{align*}(\Phi((\alpha,\beta,\gamma))\circ \Phi((\lambda,\mu,\nu)))(n,k)
&=\Phi((\alpha,\beta,\gamma))(\mu(n)\lambda(kK'),c_\nu(k))\\
&=(\beta(\mu(n)\lambda(kK'))\alpha(c_\nu(k)K'),c_\gamma(c_\nu(k)))\\
&=(\beta(\mu(n)\lambda(kK'))\alpha(c_\nu(k)K'),c_{\gamma\nu}(k))
\end{align*} 
and  
\begin{align*}\Phi((\alpha,\beta,\gamma)*(\lambda,\mu,\nu))(n,k)
&=\Phi((\alpha\cdot (\beta\circ\lambda),\beta\circ \mu,\gamma\nu))(n,k)\\
&=((\beta\circ \mu)(n)\cdot \alpha(kK')\cdot (\beta \circ \lambda)(kK'),c_{\gamma \nu}(k))\\
&=(\beta(\mu(n)\lambda(kK'))\alpha(kK'),c_{\gamma\nu}(k)).\end{align*} 
Note that $
c_{\nu}(k)K'=kK'$, so \[(\Phi((\alpha,\beta,\gamma))\circ \Phi((\lambda,\mu,\nu)))(n,k)=\Phi((\alpha,\beta,\gamma)*(\lambda,\mu,\nu))(n,k),\] and thus $\Phi$ is a homomorphism. One quickly verifies that $\ker \Phi$ is trivial, giving the injectivity of $\Phi$. 

As $\Phi$ is an injective homomorphism, $(\Hom(K^{\ab},Z(N))\rtimes \Aut(N))\times K$ is isomorphic to a subgroup of $\Aut(G)$ of order $|\Hom(K^{\ab},Z(N))|\cdot |\Aut(N)| \cdot |K|$. By \cite[Theorem 3.2]{BidwellProducts}, since $N$ and $K$ have no factors in common, \begin{align*} 
|\Aut(G)|
&=|\Hom(K,Z(N))|\cdot |\Hom(N,Z(K))|\cdot |\Aut(N)|\cdot |\Aut(K)|\\
&=|\Hom(K^{\ab},Z(N))|\cdot |\Hom(N^{\ab},Z(K))|\cdot |\Aut(N)|\cdot |\Aut(K)|.\end{align*} As $K\cong \Aut(K)$, we have $|\Aut(K)|=|K|$. Furthermore, as $K$ is centerless, $|\Hom(N,Z(K))|=1.$ Thus \[|\Aut(G)|=|\Hom(K^{\ab},Z(N))|\cdot |\Aut(N)|\cdot |K|=|\im\Phi|,\] so $\Phi$ is surjective and thus an isomorphism.
\end{proof}

\begin{lemma}
\label{div}
Let $A$ be a nontrivial abelian group. If $\Aut(A)$ is isomorphic to a (possibly trivial) direct factor of $A$, then $A\cong C_2$ or $A\cong C_6$. \end{lemma}

\begin{proof} Observe that as $\Aut(A)$ is abelian, $A$ must be cyclic. Let $n=|A|$ and suppose $n>2$. Since $\Aut(A)\cong (\bbZ/n\bbZ)^{\times}$, we have $|\Aut(A)|=\phi(n)$. Furthermore, as $\Aut(A)$ is isomorphic to a direct factor of $A$ by assumption, we have $|\Aut(A)|$ divides $|A|$. Thus $\phi(n) \mid n$. It is a standard result that such a positive integer $n$ is of the form $n=2^{e_1}3^{e_2}$ for some $e_1 \in \bbZ_{\geq 1}$ and $e_2 \in \bbZ_{\geq 0}$. As $A$ is cyclic, if $\Aut(A)$ is isomorphic to a direct factor of $A$, we have \[\gcd\biggl(|\Aut(A)|,\dfrac{|A|}{|\Aut(A)|}\biggr)=\gcd\biggl(\phi(n),\dfrac{n}{\phi(n)}\biggr)=1.\]

If $e_2=0$, then $n=2^{e_1}$, so \[\gcd\biggl(\phi(n),\dfrac{n}{\phi(n)}\biggr)=\gcd(2^{e_1-1},2)=2^{\min(e_1-1,1)}.\] We then have $2^{\min(e_1-1,1)}=1$, so $e_1-1=0$, and thus $n=2$, giving $A\cong C_2$. 

If  $e_2\geq 1$, then \[\gcd\biggl(\phi(n),\dfrac{n}{\phi(n)}\biggr)=\gcd(2^{e_1}3^{e_2-1},3)=3^{\min(e_2-1,1)}.\] We therefore find that $n=3\cdot 2^{a_1}$. Thus $A\cong C_3 \times C_{2^{a_1}}$. Observe that $\Aut(C_3)\times \Aut(C_{2^{e_1}})$ is isomorphic to a subgroup of $A$, and therefore is cyclic. As $\Aut(C_3)\cong C_2$, $\Aut(C_{2^{e_1}})$ must have odd order, so $e_2=1$ and thus $n=6$, giving $A\cong C_6$

For $A\cong C_2$, $\Aut(A)$ is trivial, and thus trivially a direct factor of $A$. For $A\cong C_6$, $\Aut(A)\cong C_2$ is a direct factor of $A\cong C_3\times C_2$, and thus the result holds. \end{proof}

\begin{thm}\label{abthm} Let $1\rightarrow A \rightarrow G \rightarrow K \rightarrow 1$ be a central extension with $A$ nontrivial and $K$ centerless. Then $G$ is stable if and only if $G\cong C_2 \times K$, $K$ is complete, and $K^{\ab}$ has a nontrivial cyclic Sylow $2$-subgroup.
\end{thm}

\begin{proof} Assume that $G$ is stable. As the extension is central, the outer action of $G$ on $A$ is trivial. As $A$ is abelian, it is nilpotent, so $G\cong A \times K$ and $K$ is complete by Lemma \ref{split}.

We claim $A\cong \Hom(K^{\ab},A)\times \Aut(A)$. To this end, we will first show $A\cong B$ where $B\cong \Hom(K^{\ab},A)\rtimes \Aut(A)$, with $\Aut(A)$ acting on $\Hom(K^{\ab},A)$ by post-composition. As $K$ is centerless, $K$ has no nontrivial abelian direct factors. Furthermore, since $A$ is abelian, $K$ and $A$ share no direct factors. Thus, by Lemma \ref{decomp}, \[ A\times K \cong G \cong \Aut(G) \cong (\Hom(K^{\ab},A)\rtimes \Aut(A))\times K.\] As the direct product is cancellative for finite groups by \cite{FiniteCancellationGroups}, $A\cong \Hom(K^{\ab},A)\rtimes \Aut(A)$. Further, as $A$ is abelian, we have \[A\cong \Hom(K^{\ab},A)\rtimes \Aut(A)\cong \Hom(K^{\ab},A)\times \Aut(A).\] Thus, by Lemma \ref{div}, $A\cong C_2$ or $A\cong C_6$

Assume towards contradiction $A\cong C_6$. Then $\Aut(A)=\langle \varphi\rangle$ where $\varphi$ is the inversion map. Thus \begin{align*}|\Hom(K^{\ab},A)|&=|A|/|\Aut(A)|=3,\end{align*} so $\Hom(K^{\ab},A)=\langle \psi \rangle\cong C_3$ for some homomorphism $\psi:K^{\ab}\rightarrow A$. Observe that $\varphi \circ \psi$ is the inverse of $\psi$ in $\Hom(K^{\ab},A)$, so the post-composition action of $\Aut(A)$ on $\Hom(K^{\ab},A)$ is inversion. Thus $C_6\cong A\cong \Hom(K^{\ab},A)\rtimes \Aut(A)$ is the nonabelian group of order six, which is clearly a contradiction. As such, $A\cong C_2$, so $G\cong C_2 \times K$. 
We must then have $|\Aut(A)|=1$, and so $
|\Hom(K^{\ab},A)|=2$. One therefore sees that $K^{\ab}$ has a nontrivial cyclic Sylow 2-subgroup. Leveraging Lemma \ref{decomp}, one quickly obtains the converse. \end{proof}

\section{Extensions by Groups of Nilpotency Class Two}

Groups of nilpotency class two, arising as central extensions of abelian groups, serve as a generalization of the abelian groups. As such, extensions of centerless groups by groups of nilpotency class two constitute a natural next step in the search for stable groups with nontrivial central elements and outer automorphisms. In particular, extensions inducing a trivial outer action on the kernel generalize central extensions, ensuring centrality of the center of the kernel by Lemma \ref{extchar}. With this in mind, in this section we classify all stable extensions of centerless groups by groups of nilpotency class two whose induced outer action on the kernel is trivial.

\begin{lemma}
\label{noncyclic} Let $G$ be an abelian $p$-group with no repeated direct factors. If $G$ is not cyclic, then $|G|$ divides $ |\Aut(G)|$.
\end{lemma}
\begin{proof}
 Since $G$ is not cyclic, $G$ is the product of $n$ pairwise distinct nontrivial cyclic groups of $p$-power order for some $n \in \bbZ_{> 1}$. Thus $G\cong H\times C_{p^{a_{n}}}$ where $H=\displaystyle\prod_{i=1}^{n-1}C_{p^{a_i}}$ with $1< a_1 < \cdots < a_n$. By \cite[Theorem 3.2]{BidwellProducts}, 
 \begin{align*}
 |\Aut(G)|&=|\Aut(H\times C_{p^{a_{n}}})|\\&=|\Hom(H,C_{p^{a_{n}}})|\cdot |\Hom(C_{p^{a_{n}}},H)|\cdot|\Aut(H)|\cdot |\Aut(C_{p^{a_{n}}} )|. 
 \end{align*} 
 By the additivity of the $\Hom$ functor we have \begin{align*}
\Hom(H,C_{p^{a_{n}}})&=\Hom\biggl(\displaystyle\prod_{i=1}^{n-1}C_{p^{a_i}},C_{p^{a_{n}}}\biggr)
\cong \displaystyle\prod_{i=1}^{n-1}\Hom(C_{p^{a_i}},C_{p^{a_{n}}})\\&\cong \displaystyle\prod_{i=1}^{n-1} C_{p^{a_i}}= H
\end{align*} 
as $\gcd(p^{a_i},p^{a_{n}})=p^{a_i}$ for all $i \in \{1,\dots, n-1\}$. Thus $|\Hom(H,C_{p^{a_{n}}})|=|H|$, and similarly $|\Hom(C_{p^{a_{n}}},H)|=|H|$. Since $|\Aut(C_{p^{a_{n}}})|=(p-1)p^{a_{n}-1}$, it follows that 
\begin{align*}
\upsilon_p(|\Aut(G)|)&=\upsilon_p(|H|^2\cdot |\Aut(H)|\cdot |\Aut(C_{p^{a_{n}}})|)\geq 2\upsilon_p(|H|)+a_{n}-1.\end{align*} Note $H$ is itself a nontrivial $p$-group, so $p$ divides $|H|$, and thus \begin{align*}2\upsilon_p(|H|)+a_{n}-1\geq \upsilon_p(|H|)+a_{n}=\upsilon_p(|H|\cdot |C_{p^{a_n}}|)=\upsilon_p(|G|).
\end{align*} 
As $G$ is a $p$-group, we therefore conclude that $|G|$ divides $|\Aut(G)|$.
\end{proof}

\begin{lemma}\label{pure} Let $N$ be a group of nilpotency class two. If $|\Aut(N)|$ divides $|N|$, then each Sylow $p$-subgroup of $N$ is abelian or purely nonabelian.
\end{lemma}
\begin{proof}
Let $N_p$ be a Sylow $p$-subgroup of $N$. Suppose towards contradiction $N_p$ is neither abelian nor purely nonabelian. Then we may decompose $N_p\cong A_p \times B_p$ such that $B_p$ is purely nonabelian and $A_p$ is abelian, with $A_p$ and $B_p$  both nontrivial. By \cite[Theorem 3.2]{BidwellProducts}, \begin{align*}\upsilon_p(|\Aut(N_p)|)&=\upsilon_p(|\Hom(A_p,Z(B_p))|)+\upsilon_p(|\Hom(B_p,A_p)|)\\&+\upsilon_p(|\Aut(A_p)|)+\upsilon_p(|\Aut(B_p)|).\end{align*} Note that $Z(B_p)$ is nontrivial, so $\upsilon_p(|\Hom(A_p,Z(B_p))|)\geq 1$. As $B_p$ is necessarily nilpotent of class two, $B_p^{\ab}$ is nontrivial, so \[\upsilon_p(|\Hom(B_p,A_p)|)=\upsilon_p(|\Hom(B_p^{\ab},A_p)|)\geq 1.\] 

As $B_p$ is nilpotent of class two, $\upsilon_p(|B_p|)\leq \upsilon_p(|\Aut(B_p)|)$ by \cite{FaudreeNil2}. If $A_p$ is cyclic, we find $\upsilon_p(|\Aut(A_p)|)=\upsilon_p(\phi(|A_p|))=\upsilon_p(|A_p|)-1$. Otherwise, as $\Aut(C_{p^a}\times C_{p^a})\cong \GL_2(\bbZ/p^a\bbZ)\twoheadrightarrow \GL_2(\bbZ/p\bbZ)$ is not nilpotent for all $a\geq 1$, $A_p$ has no repeated direct factors. Thus, by Lemma \ref{noncyclic}, we have $\upsilon_p(|\Aut(A_p)|)\geq \upsilon_p(|A_p|)> \upsilon_p(|A_p|)-1$. In both cases, we have 
\begin{align*}
\upsilon_p(|\Aut(N_p)|)&\geq 1+1+(\upsilon_p(|A_p|)-1)+\upsilon_p(|B_p|)\\&=1+\upsilon_p(|A_p|\cdot |B_p|)=1+\upsilon_p(|N_p|).\end{align*} As $\Aut(N_p)$ is isomorphic to a direct factor of $\Aut(N)$, we have \[\upsilon_p(|\Aut(N_p)|)\leq \upsilon_p(|\Aut(N)|)\leq \upsilon_p(|N|)=\upsilon_p(|N_p|),\] which is clearly a contradiction. Thus, $N_p$ is either abelian or purely nonabelian.
\end{proof}

\begin{lemma}\label{D4} Suppose $N\times K$ is stable, where $N$ is a nontrivial nilpotent group of class two and $K$ is complete. Then $N\cong D_4$. 
\end{lemma}
\begin{proof} 
 As $N\times K$ is stable, \begin{align*}N\times K \cong \Aut(N\times K)\cong (\Hom(K^{\ab},Z(N))\rtimes \Aut(N))\times K\end{align*} by Lemma \ref{decomp}. As the direct product is cancellative for finite groups by \cite{FiniteCancellationGroups}, \[N\cong \Hom(K^{\ab},Z(N))\rtimes \Aut(N),\] so $|\Aut(N)|$ divides $|N|$. Let $\mP$ be the set of primes dividing $|N|$. For each prime $p \in \mP$, choose a Sylow $p$-subgroup of $N$ and denote it $N_p$. Let \[\mA=\{p \in \mP \mid N_p \textrm{ is abelian}\}\] and \[\mB=\{p \in \mP \mid N_p \textrm{ is purely nonabelian}\},\] and note that $\mA \sqcup \mB = \mP$ by Lemma \ref{pure}. Let $A=\displaystyle\prod_{p \in \mA}N_{p}$ and $B=\displaystyle\prod_{p \in \mB}N_p$. Suppose towards contradiction there exists a prime $p \in \mA$ such that $p$ divides $|\Aut(B)|$. Then there is an element $\varphi \in \Aut(B)$ with order $p$. Let $\varphi_N$ be the image of $\varphi$ in $\Aut(N)$ induced by the natural isomorphism $N\cong A\times B$. Consider $\varphi_N$ embedded in $\Hom(K^{\ab},Z(N))\rtimes \Aut(N)\cong N \cong A\times B$. As $p \in \mA$ and the image of $\varphi_N$ in $A\times B$ has order $p$, the image of $\varphi_N$ in $A\times B$ projects trivially onto $B$. As such, the image of $\varphi_N$ in $A\times B$ is central, so $\varphi_N \in Z(\Hom(K^{\ab},Z(N))\rtimes \Aut(N))$. Thus, $\varphi_N \in Z(\Aut(N))$, and so $\varphi \in Z(\Aut(B))$. As $\varphi$ is nontrivial, there exists a prime $q \in \mB$ such that the projection of $\varphi$ onto $\Aut(N_q)$, denoted $\varphi_q$, has order $p$. Observe that $\varphi_q\in Z(\Aut(N_q))\leq \Aut_c(N_q)$. As $N_q$ is purely nonabelian, $\Aut_c(N_q)$ is a $q$-group by \cite[Corollary 2]{AdneyCentAut}. As $\varphi_q \in \Aut_c(N_q)$ has order $p$, $p$ divides $|\Aut_c(N_q)|$, which is clearly a contradiction as $p$ and $q$ are coprime by Lemma \ref{pure}. We therefore find that $\gcd(p,|\Aut(B)|)=1$ for all $p \in \mA$.

 For all primes $p \in \mB$, we have \[\upsilon_p(|B|)\leq \upsilon_p(|\Aut(B)|)\leq \upsilon_p(|\Aut(N)|)\leq \upsilon_p(|N|)=\upsilon_p(|B|)\] by \cite{FaudreeNil2} and the definition of $\mB$. Thus, as all primes dividing $|\Aut(B)|$ are members of $\mB$, we have $|\Aut(B)|=|B|$. 
 
As \[\displaystyle\prod_{p \in \mB}|\Aut(N_p)|=|\Aut(B)|=|B|=\displaystyle\prod_{p \in \mB}|N_p|\] and $|N_p|\leq |\Aut(N_p)|$ for all $p \in \mB$ by \cite{FaudreeNil2}, we have $|N_p|=|\Aut(N_p)|$ for all $p \in \mB$. Therefore, for all $p \in \mB$, $\Aut(N_p)$ is isomorphic to a Sylow $p$-subgroup of $ \Hom(K^{\ab},Z(N))\rtimes \Aut(N)\cong N$. Thus $N_p$ is stable for all primes $p \in \mB$. As $N$ is nilpotent of class two, $\mB\neq \emptyset$. By \cite{CutoloD4}, the only stable $p$-group of nilpotency class two is the dihedral group of order eight, so $B\cong D_4$.

 Observe that \begin{align*}|A|\cdot |B|&=|\Hom(K^{\ab},Z(N))|\cdot |\Aut(N)|\\&=|\Hom(K^{\ab},Z(N))|\cdot |\Aut(A)|\cdot |\Aut(B)|.\end{align*} As $B\cong D_4$ is stable, we therefore have $|A|=|\Hom(K^{\ab},Z(N))|\cdot |\Aut(A)|$. Since $2 \in \mB$ and $\mA\cap \mB=\emptyset$, $A$ has odd order. If $A$ is nontrivial, then the inversion automorphism of $A$ has order two, so $\Aut(A)$ has even order, which is clearly a contradiction. Thus $A$ is trivial, so $N\cong D_4$.
\end{proof}

\begin{cor} If $N$ is a stable group of nilpotency class two, then $N\cong D_4$.
\end{cor}

\begin{proof} Viewing $N$ as isomorphic to a direct product of itself with the trivial group, Lemma \ref{D4} immediately gives the desired result.
\end{proof}

\begin{thm} \label{nilthm} Let $1\rightarrow N\rightarrow G\rightarrow K\rightarrow 1$ be an extension inducing a trivial outer action of $G$ on $N$, such that $N$ is nilpotent of class two and $K$ is centerless. Then $G$ is stable if and only if $G\cong D_4\times K$, $K$ is complete, and $|K^{\ab}|$ is odd.
\end{thm}
\begin{proof}
Suppose $G$ is stable. By Lemma \ref{split}, $G\cong N\times K$ and $K$ is complete. By Lemma \ref{D4} $N\cong D_4$, so \begin{align*}\Aut(N\times K)&\cong (\Hom(K^{\ab},Z(N))\rtimes \Aut(N))\times K \\&\cong (\Hom(K^{\ab},Z(D_4))\rtimes \Aut(D_4))\times K\end{align*} by Lemma \ref{decomp}. As $D_4$ is stable, we have $|\Aut(D_4)|=|D_4|$, and thus \[|\Hom(K^{\ab},C_2)|=|\Hom(K^{\ab},Z(D_4))|=1.\] As such, $|K^{\ab}|$ must be odd.

Finally, suppose $K$ is complete, $|K^{\ab}|$ is odd, and $G\cong D_4\times K$. Leveraging Lemma \ref{decomp}, one quickly obtains the converse.
\end{proof}

\bibliographystyle{amsplain}
 \bibliography{bibliography}
\end{document}

%% file: mymakros.tex
\newcommand{\Inn}{\textrm{Inn}}

\newcommand{\im}{\textrm{im}\,}
\newcommand{\norm}{\trianglelefteq}

\DeclareMathOperator{\ab}{ab}

\DeclareMathOperator{\Aut}{Aut}

\DeclareMathOperator{\GL}{GL}

\DeclareMathOperator{\Hom}{Hom}
\DeclareMathOperator{\id}{id}

\newcommand{\mA}{\mathcal{A}}
\newcommand{\mB}{\mathcal{B}}

\newcommand{\mP}{\mathcal{P}}

\newcommand{\bbZ}{\mathbb{Z}}

\newcommand{\bmat}{\setstretch{1}\left[ \begin{matrix}}
\newcommand{\emat}{\end{matrix} \right]\setstretch{2}}

\newcommand{\tm}



%% file: settings.tex
\usepackage[all]{xy}
\usepackage{amsthm}
\usepackage{amssymb, amsfonts}
\SelectTips{cm}{10}\UseTips

\theoremstyle{plain}
\newtheorem{thm}{Theorem}

\newtheorem{cor}[thm]{Corollary}
\newtheorem{lemma}[thm]{Lemma}

\theoremstyle{definition}

\numberwithin{thm}{section}
\numberwithin{equation}{section}